\documentclass[reqno, 11pt]{amsart}
\usepackage{amssymb}
\usepackage{amsmath}
\usepackage[mathscr]{euscript}
\usepackage{graphicx}
\usepackage{comment}
\usepackage{enumitem}

\usepackage[dvipsnames]{xcolor}
\usepackage{hyperref} 
\definecolor{halfgray}{gray}{0.55} 
\definecolor{webgreen}{rgb}{0,0.5,0}
\definecolor{webbrown}{rgb}{.6,0,0} \hypersetup{%
  colorlinks=true, linktocpage=true, pdfstartpage=3,
  pdfstartview=FitV,%
  breaklinks=true, pdfpagemode=UseNone, pageanchor=true,
  plainpages=false, bookmarksnumbered, bookmarksopen=true,
  bookmarksopenlevel=1,%
  hypertexnames=true,
  pdfhighlight=/O,
  urlcolor=webbrown, linkcolor=RoyalBlue,
  citecolor=webgreen, 
  pdftitle={},%
  pdfauthor={},%
  pdfsubject={2000 MAthematical Subject Classification: Primary:},%
  pdfkeywords={},%
  pdfcreator={pdfLaTeX},%
  pdfproducer={LaTeX with hyperref}%
}

\usepackage[
  textwidth=15cm,
  textheight=22.7cm,
  centering]{geometry} 

\numberwithin{equation}{section}

\theoremstyle{plain}

\newtheorem{theorem}{Theorem}

\newtheorem{proposition}[theorem]{Proposition}

\newtheorem{corollary}[theorem]{Corollary}

\theoremstyle{definition}

\theoremstyle{remark}

\theoremstyle{plain}

\newtheorem*{conjecture}{Generalized Hyperbolicity Conjecture for Shadowing}

\theoremstyle{plain}

\newtheorem{maintheorem}{Theorem}

\newtheorem{claim}{Claim}
\newcommand{\T}{\mathbb T}
\newcommand{\C}{\mathbb C}
\newcommand{\N}{\mathbb N}

\newcommand{\ran}{\operatorname{ran}}

\begin{document}
\title[The Generalized Hyperbolicity Conjecture]{Resolving the Generalized Hyperbolicity Conjecture for Shadowing}

\author[M. Pituk]{Mih\'aly Pituk}
\address{Department of Mathematics, University of Pannonia, Egyetem \'ut 10, 8200 Veszpr{\'e}m, Hungary;
HUN--REN--ELTE Numerical Analysis and Large Networks Research Group, Budapest, Hungary}
\email{pituk.mihaly@mik.uni-pannon.hu}



\begin{abstract}
It is known that generalized hyperbolicity implies the shadowing property for invertible bounded linear operators on a Banach space. Whether the converse holds has been a central open problem in linear dynamics and has been conjectured to have a positive answer. We show that this conjecture fails on general Banach spaces by constructing a counterexample, whereas it holds on separable Hilbert spaces. The distinction is explained by the gap that may occur on Banach spaces between surjectivity and right invertibility, a gap that disappears on Hilbert spaces. The main ingredients of the proofs are a recent spectral characterization of shadowing in terms of the surjective spectrum and a new characterization of generalized hyperbolicity in terms of right resolvent functions near the unit circle.
 \end{abstract}

\maketitle

\section{Introduction and main results}\label{sec:intro}
In this paper, we resolve one of the central open problems in linear dynamics:
whether the shadowing property implies generalized hyperbolicity for
invertible bounded linear operators on a Banach space.
Due to its importance, this problem has received recurring attention in the literature~(see \cite[Problem~5.0.3]{Dani}, \cite[Problem~F]{BCDFP}, \cite[Remark~18]{LM}, \cite[Question~3.10]{Mo}).
For precision, we first introduce the necessary notation and basic concepts.

Let~$X$ be a complex Banach space. We denote by~$L(X)$ the Banach
algebra of bounded linear operators on~$X$, equipped with the operator
norm, and by~$GL(X)$ the group of invertible elements of~$L(X)$. 

\noindent 
An operator $T\in L(X)$ is called \emph{invertible} if there exists $S\in L(X)$ such that $TS=ST=I$, where~$I=I_X$ denotes the identity operator on~$X$.
We say that an operator~$T\in L(X)$ is \emph{right invertible} (\emph{left invertible}) if there exists $S\in L(X)$ such that $TS=I$ ($ST=I$). The  \emph{spectrum}, \emph{right spectrum}, \emph{left spectrum} and \emph{surjective spectrum} of~$T$ are defined by
\begin{align*}
\sigma(T)&=\{\,\lambda\in\mathbb C\mid\text{$\lambda I-T$ is not invertible}\,\},\\
\sigma_r(T)&=\{\,\lambda\in\mathbb C\mid\text{$\lambda I-T$ is not right invertible}\,\},\\
\sigma_l(T)&=\{\,\lambda\in\mathbb C\mid\text{$\lambda I-T$ is not left invertible}\,\},\\
\sigma_s(T)&=\{\,\lambda\in\mathbb C\mid\text{$\lambda I-T$ is not surjective}\,\}.
\end{align*}
The \emph{resolvent set}, \emph{right resolvent set} and \emph{left resolvent set} of~$T$ are given by
\begin{equation*}
\rho(T)=\mathbb C\setminus\sigma(T),\qquad\rho_r(T)=\mathbb C\setminus\sigma_r(T),\qquad \rho_l(T)=\mathbb C\setminus\sigma_l(T).
\end{equation*}
The \emph{spectral radius} of~$T$ is defined by $r(T)=\sup_{\lambda\in\sigma(T)}|\lambda|$. As usual, $\ker T$ and $\operatorname{ran} T$ denote the kernel and the range of~$T$, respectively.

It is known (see, e.g., \cite[Theorem~2.12]{B}) that an operator~$T\in L(X)$ on a complex Banach space~$X$ is right invertible if and only if $T$ is surjective and $\ker T$ is complemented in~$X$. Consequently, $
\sigma_s(T)\subseteq\sigma_r(T)$,
and the inclusion may be strict on Banach spaces. If~$X$ is a complex Hilbert space, then every closed subspace is complemented in~$X$, hence $\sigma_s(T)=\sigma_r(T)$.

\emph{Shadowing} is one of the fundamental concepts in the theory of dynamical systems. Recall its definition in the setting of linear dynamics~\cite{Pal}, \cite{Pil}. Given $\delta>0$, a sequence $(x_n)_{n\in\mathbb Z}$ in~$X$ is called a \emph{$\delta$-pseudotrajectory} of an operator~$T\in GL(X)$ if 
\begin{equation*}
\|x_{n+1}-Tx_n\|\leq\delta\qquad\text{for all $n\in\mathbb Z$}.	
\end{equation*}
An operator $T\in GL(X)$ is said to have the \emph{shadowing property} if,  for every $\epsilon>0$, there exists $\delta>0$ such that every $\delta$-pseudotrajectory $(x_n)_{n\in\mathbb Z}$ of~$T$ is $\epsilon$-shadowed by an exact trajectory of~$T$; that is, there exists $v\in X$ such that
\begin{equation*}
\|x_n-T^n v\|\leq\epsilon\qquad\text{for all $n\in\mathbb Z$.} 
\end{equation*}

In a recent paper~\cite{DP}, we established the following spectral characterization of the shadowing property. 
\begin{theorem}[{\cite[Theorem~2.2]{DP}}]\label{thm:specban}
Let $X$ be a complex Banach space and let $T\in GL(X)$. Then~$T$ has the shadowing property if and only if $\sigma_s(T)\cap\mathbb T=\emptyset$.
\end{theorem}

As noted above, if~$X$ is a Hilbert space, then $\sigma_s(T)=\sigma_r(T)$. Thus, Theorem~\ref{thm:specban} yields the following corollary.

\begin{corollary}[{\cite[Theorem~1.7]{P}}]\label{cor:spechilb}
Let $X$ be a complex Hilbert space and let $T\in GL(X)$. Then~$T$ has the shadowing property if and only if $\sigma_r(T)\cap\mathbb T=\emptyset$.
\end{corollary}

Besides shadowing, another fundamental concept in dynamical systems is \emph{hyperbolicity}.
We say that an operator $T\in GL(X)$ is \emph{hyperbolic} if $\sigma(T)\cap\mathbb T=\emptyset$, where~$\mathbb T$ denotes the unit  circle in the complex plane. By the Riesz decomposition theorem,  $T\in GL(X)$ is hyperbolic if and only if $X$ admits a direct-sum decomposition $X=M\oplus N$, where $M$ and~$N$ are closed subspaces of~$X$ such that 
\[
T(M)=M,\qquad T(N)=N,
\]
and the restrictions $T|_M\in L(M)$ and $T^{-1}|_N\in L(N)$ have spectral radius less than one; that is, 
\begin{equation*}
r(T|_M)<1\quad\text{and}\quad r(T^{-1}|_N)<1.
\end{equation*}
Generalized hyperbolicity extends classical hyperbolicity by weakening the invariance conditions. More precisely, $T\in GL(X)$ is called \emph{generalized hyperbolic} if $X$ admits a direct-sum decomposition $X=M\oplus N$, where $M$ and~$N$ are closed subspaces of~$X$ such that 
\[
T(M)\subseteq M,\qquad T^{-1}(N)\subseteq N,
\]
and
\[
r(T|_M)<1,\qquad r(T^{-1}|_N)<1.
\]
This class of operators first appeared in the work of Bernardes et al.~\cite{Ber}. The terminology ``generalized hyperbolic'' was introduced later by Cirilo et al.~\cite{Cir}, who established several important dynamical properties of generalized hyperbolic operators.

It has long been known that invertible hyperbolic operators on Banach
spaces have the shadowing property and, for invertible operators on finite-dimensional spaces, the shadowing property is equivalent to hyperbolicity.  In~\cite{Ber}, Bernardes et al. constructed a nonhyperbolic operator with
the shadowing property, thereby settling a long-standing open problem
in the literature. A key ingredient in their construction was the following result (see also \cite[Theorem~1]{BCDFP}).

\begin{theorem}[{\cite[Theorem~A]{Ber}}]\label{thm:specuniexp}
Let $X$ be a complex Banach space and let $T\in GL(X)$. If~$T$ is generalized hyperbolic, then~$T$ has the shadowing property.
\end{theorem}

Whether the converse of Theorem~\ref{thm:specuniexp} holds has been one of the central open problems in linear dynamics. In \cite[Theorem~18]{BeMe},
Bernardes and Messaoudi characterized the scalar weighted shifts on
$\ell^p$, $1\leq p<\infty$, and on $c_0$ that have the shadowing
property. Their characterization implies that, for classical bilateral backward weighted shifts, the shadowing property is equivalent to generalized hyperbolicity (see \cite[Theorem~2.4.5]{Dani}). D'Aniello, Darji and Maiuriello~\cite{Dani} established the same equivalence for a large class of composition operators on $L^p$-spaces and posed the question of whether
the shadowing property implies generalized hyperbolicity (see \cite[Problem~5.0.3]{Dani}). 
In several related works, the expected affirmative answer was explicitly mentioned as a conjecture (see, e.g., \cite{AMV}, \cite{LM1}, \cite{LM2}).
Thus, we state it as follows.

\begin{conjecture}
\label{conj:generalized-hyperbolicity}
Let \(X\) be a complex Banach space and let \(T\in GL(X)\). If \(T\) has the
shadowing property, then \(T\) is generalized hyperbolic.
\end{conjecture}

The purpose of the present paper is to settle this conjecture. We show that it is false on general Banach spaces, but true on separable Hilbert spaces.
\begin{maintheorem}\footnote{After the manuscript had been completed, the author became aware that Theorem A had also been proved independently by Messaoudi et al.~\cite[Theorem~3.9]{MNST}.}\label{thm:A}
There exist a complex Banach space~$X$ and an operator $T\in GL(X)$
such that $T$ has the shadowing property but is not generalized
hyperbolic.
\end{maintheorem}

\begin{maintheorem}\footnote{In~\cite[Theorem~3.5]{MNST}, Messaoudi et al. proved that, for invertible operators on Hilbert spaces, the shadowing property is equivalent to pseudo-hyperbolicity. However, pseudo-hyperbolicity is weaker than generalized hyperbolicity, since it does not require a decomposition $X=M\oplus N$ into closed subspaces satisfying $T(M)\subseteq M$ and $T^{-1}(N)\subseteq N$. Thus, it does not settle the generalized hyperbolic conjecture.}
\label{thm:B}
Let $X$ be a separable complex Hilbert space and let $T\in GL(X)$.
Then $T$ has the shadowing property if and only if it is generalized
hyperbolic.
\end{maintheorem}
It remains open whether the separability assumption in Theorem~\ref{thm:B} can be removed.

The proofs of Theorems~\ref{thm:A} and~\ref{thm:B} are given in Section~3.
The main structural tool is a new characterization of generalized hyperbolicity in terms of the existence of a right resolvent function on an open annulus containing the unit circle. 
Before stating this characterization, we recall the notion of a right resolvent function from~\cite[Sec.~9.2]{A}.
 
Let~$X$ be a complex Banach space, let $T\in L(X)$, and let $\Omega\subseteq\rho_r(T)$ be an open set. We say that \emph{$T$ admits a right resolvent function  on~$\Omega$} if there exists a continuous function $R\colon\Omega\rightarrow L(X)$ satisfying
\begin{equation}\label{eq:hfri}
	(\lambda I-T)R(\lambda)=I\qquad\text{for all $\lambda\in\Omega$}
\end{equation}
and the \emph{resolvent equation}
\begin{equation}\label{eq:reseq}
	R(\lambda)-R(\mu)=(\mu-\lambda)R (\lambda)R(\mu)\qquad\text{for all $\lambda,\mu\in\Omega$}.
\end{equation}
As noted in \cite[Sec.~3.1.3]{H}, every right resolvent function~$R$ is holomorphic on~$\Omega$.
Allan~\cite{All} (see also ~\cite{Iva}) has shown that, for every open and connected set $\Omega\subseteq\rho_r(T)$, there exists a holomorphic family of right inverses on~$\Omega$. By a \emph{holomorphic family of right inverses on~$\Omega$}, we mean a holomorphic function $R\colon\Omega\rightarrow L(X)$ satisfying~\eqref{eq:hfri}. In contrast to the classical resolvent function 
\begin{equation*}
R(\lambda)=(\lambda I-T)^{-1},\qquad \lambda\in\rho(T),
\end{equation*}
a holomorphic family of right inverses need not satisfy the resolvent equation~\eqref{eq:reseq}. For separable Hilbert spaces, Apostol et al.~\cite{A} established the following sufficient condition for the existence of a right resolvent function.

\begin{proposition}[{\cite[Proposition~9.17]{A}}]\label{prop:sapostol}
Let $X$ be a separable complex Hilbert space and let $T\in L(X)$. If $\Omega\subset\rho_r(T)$ is an open set such that $\overline{\Omega}\subset\rho_r(T)$, where $\overline{\Omega}$ denotes the closure of $\Omega$ in~$\mathbb C$, and 
\begin{equation*}
\operatorname{nul}(\lambda I-T):=\dim\ker(\lambda I-T)\quad\text{is constant for $\lambda\in\Omega$},
\end{equation*}
 then~$T$ admits a right resolvent function on~$\Omega$.
\end{proposition}

For $\epsilon\in(0,1)$, set
$$
\mathbb T_\epsilon
=\{\,\lambda\in\mathbb C\mid 1-\epsilon<|\lambda|<1+\epsilon\,\},
$$
the open annulus around the unit circle with inner radius $1-\epsilon$ and outer radius
$1+\epsilon$.

We can now state the following new characterization of generalized hyperbolicity.

\begin{maintheorem}\label{thm:C}

Let~$X$ be a complex Banach space and let $T\in GL(X)$.
Then $T$ is generalized hyperbolic if and only if there exists $\epsilon\in(0,1)$ such that 
$\mathbb T_\epsilon\subseteq\rho_r(T)$ and $T$ admits a right resolvent function on~$\mathbb T_\epsilon$.
\end{maintheorem}

The proof of Theorem~C is given in Section~2. Theorem~C yields the following corollary.

\begin{corollary}\label{cor:ghrh}
Let~$X$ be a complex Banach space and let $T\in GL(X)$.
If $T$ is generalized hyperbolic, then $\sigma_r(T)\cap\mathbb T=\emptyset$. 
	
\end{corollary}

Corollary~\ref{cor:ghrh}, together with the inclusion $\sigma_s(T)\subseteq \sigma_r(T)$
and Theorem~\ref{thm:specban}, gives another proof of the fact that every generalized hyperbolic operator has the shadowing property.
More importantly, Theorem~\ref{thm:specban} and Corollary~\ref{cor:ghrh} reveal a natural idea for constructing a counterexample to the conjecture: if one can find an invertible operator~$T$ on a Banach space such that
\begin{equation*}
\sigma_s(T)\cap\mathbb T=\emptyset
\qquad\text{but}\qquad
\sigma_r(T)\cap\mathbb T\neq\emptyset,
\end{equation*}
then~$T$ has the shadowing property but is not generalized hyperbolic. The counterexample constructed in the proof of Theorem~A has precisely these properties.

This obstruction disappears on Hilbert spaces, where $\sigma_s(T)=\sigma_r(T)$. If $X$ is a Hilbert space and $T\in GL(X)$ has the shadowing property, then Corollary~\ref{cor:spechilb} gives $\sigma_r(T)\cap\mathbb T=\emptyset$. 
Since~$\rho_r(T)$ is open and contains the compact set~$\mathbb T$, there exists $\epsilon>0$ such that
$\overline{\mathbb T_\epsilon}\subset \rho_r(T)$.
Assume, in addition, that~$X$ is separable. In the proof of Theorem~\ref{thm:B}, we show that the function
$\operatorname{nul}(\lambda I-T)$
is locally constant for $\lambda\in\mathbb T_\epsilon$. Since the annulus $\mathbb T_\epsilon$ is connected, the nullity is therefore constant there. Proposition~\ref{prop:sapostol} implies that~$T$ admits a right resolvent function on~$\mathbb T_\epsilon$ and Theorem~\ref{thm:C} yields that~$T$ is generalized hyperbolic.

Thus, the different behavior on general Banach spaces and on separable Hilbert spaces  is related to the distinction between surjectivity and right invertibility. On general Banach spaces, the inclusion
$\sigma_s(T)\subseteq\sigma_r(T)$
may be strict, which makes the counterexample in Theorem~\ref{thm:A} possible. On Hilbert spaces the two spectra coincide, and in the separable case the result of Apostol et al.~\cite{A}, together with Theorem~\ref{thm:C}, yields the equivalence between the shadowing property and generalized hyperbolicity stated in Theorem~\ref{thm:B}.

The paper is organized as follows. In Section~2, we prove Theorem~\ref{thm:C}. In Section~3, using Theorem~\ref{thm:C}, we prove Theorems~\ref{thm:A} and~\ref{thm:B}.

\section{Proof of Theorem~\ref{thm:C}}

In this section, we give a proof of Theorem~C. Let $X$ be a complex Banach space. Throughout the paper, by a \emph{projection} on~$X$, we mean a bounded operator $P\in L(X)$ such that $P^2=P$. Recall that a subspace $E$ of~$X$ is \emph{complemented in~$X$} if there exists a projection $P\in L(X)$ such that $E=\operatorname{ran}P$.

\begin{proof} [Proof of Theorem~C]  \textit{Necessity.} Suppose that $T$ is generalized hyperbolic. We will show the existence of $\epsilon\in(0,1)$ such that  $\mathbb T_\epsilon\subseteq\rho_r(T)$ and $T$ admits a right resolvent function on~$\mathbb T_\epsilon$. Let $X=M\oplus N$ be the direct-sum decomposition from the definition of generalized hyperbolicity. 
Let $P$ be the projection of~$X$ onto~$M$ along $N$ and let
$
Q=I-P$.
Define
\begin{equation*}
	A=T|_M\in L(M)
\qquad\text{and}\qquad
B=T^{-1}|_N\in L(N).
\end{equation*}
Choose $\epsilon\in(0,1)$ so small that
\begin{equation}\label{eq:epschoice}
r(A)<1-\epsilon
\qquad\text{and}\qquad
r(B)<\frac{1}{1+\epsilon}.
\end{equation}
For $\lambda\in\mathbb T_\epsilon$, define
\begin{equation}\label{eq:rightresser}
R(\lambda)
=
\sum_{n=0}^{\infty}\lambda^{-n-1}T^n P
-
\sum_{n=0}^{\infty}\lambda^nT^{-n-1}Q.
\end{equation}
We will show that~$R$ is well-defined. Choose $a$ and~$b$ such that
\begin{equation}\label{eq:abchoice}
1-\epsilon<a<1,\qquad 1<b<1+\epsilon.
\end{equation}
From~\eqref{eq:epschoice} and Gelfand's spectral radius formula, it follows that there exists $C>0$ such that for all $n\geq0$,
$$
\|A^n\|\leq C(1-\epsilon)^n,\qquad\|B^n\|\leq C\frac{1}{(1+\epsilon)^n}.
$$
Since $T(M)\subseteq M$, we have $T^n P=A^n P$ for $n\geq0$. Similarly, $T^{-1}(N)\subseteq N$ implies $T^{-n}Q=B^nQ$ for $n\geq0$.
From this, we find for $a\leq|\lambda|\leq b$,
$$
\sum_{n=0}^\infty\|\lambda^{-n-1}T^n P\|=\sum_{n=0}^\infty|\lambda|^{-n-1}\|A^n P\|\leq\sum_{n=0}^\infty a^{-n-1}C(1-\epsilon)^n\|P\|
=\frac{C\|P\|}{a-(1-\epsilon)}
$$
and
$$
\sum_{n=0}^\infty\|\lambda^{n}T^{-n-1} Q\|=\sum_{n=0}^\infty|\lambda|^{n}\|B^{n+1} Q\|\leq\sum_{n=0}^\infty b^{n}C\frac{1}{(1+\epsilon)^{n+1}}\|Q\|
=\frac{C\|Q\|}{(1+\epsilon)-b}.
$$
Since the last estimates in both strings of inequalities are independent of~$\lambda$, by the Weierstrass $M$-test, both operator-valued series converge uniformly on every compact subannulus $a\leq|\lambda|\leq b$ of~$\mathbb T_\epsilon$. Consequently, their sums are continuous on~$\mathbb T_\epsilon$. This proves that $R\colon\mathbb T_\epsilon\rightarrow L(X)$ is a continuous function.

For $\lambda\in\mathbb T_\epsilon$, define
$$
R_M(\lambda)=\sum_{n=0}^{\infty}\lambda^{-n-1}T^nP,
\qquad
R_N(\lambda)=-\sum_{n=0}^{\infty}\lambda^nT^{-n-1}Q,
$$
so that $R=R_M+R_N$ on $\mathbb T_\epsilon$. For every $\lambda\in\mathbb T_\epsilon$, we have
\begin{align*}
(\lambda I-T)R_M(\lambda)
&=
\sum_{n=0}^{\infty}\lambda^{-n}T^nP
-
\sum_{n=0}^{\infty}\lambda^{-n-1}T^{n+1}P
=P,
\end{align*}
and
\begin{align*}
(\lambda I-T)R_N(\lambda)
&=
-\sum_{n=0}^{\infty}\lambda^{n+1}T^{-n-1}Q
+
\sum_{n=0}^{\infty}\lambda^nT^{-n}Q
=Q.
\end{align*}
Consequently,
\begin{equation}\label{eq:rightinveq}
(\lambda I-T)R(\lambda)=P+Q=I,\qquad \lambda\in\mathbb T_\epsilon.
\end{equation}
Thus, $\mathbb T_\epsilon\subseteq\rho_r(T)$ and the right-inverse identity~\eqref{eq:hfri} holds for $\Omega=\mathbb T_\epsilon$.

It remains to prove the resolvent equation~\eqref{eq:reseq}.
It is known (see \cite[Theorem~5.2--C, p.~262]{T}) that
\begin{equation*}
	(\lambda I-A)^{-1}=\sum_{n=0}^\infty\lambda^{-n-1}A^n\qquad\text{whenever $|\lambda|>r(A)$}.
\end{equation*}
By the same formula,
\begin{equation*}
	(I-\lambda B)^{-1}=\sum_{n=0}^\infty\lambda^{n}B^n\qquad\text{whenever $1>r(\lambda B)=|\lambda|r(B)$}.
\end{equation*}
From this and~\eqref{eq:rightresser}, using $T^n P=A^n P$ and $T^{-n}Q=B^n Q$ for $n\geq0$ again, we find that
\begin{equation*}
R(\lambda)=(\lambda I-A)^{-1}P-B(I-\lambda B)^{-1} Q	\qquad\text{whenever $r(A)<|\lambda|$ and $|\lambda|r(B)<1$}.
\end{equation*}
We claim that
\begin{equation}\label{eq:commonrange}
	\operatorname{ran} R(\lambda)=M+T^{-1}(N)\qquad\text{for all $\lambda\in\mathbb T_\epsilon$}.
\end{equation}
Let $\lambda\in\mathbb T_\epsilon$. Every $x\in X$ has a unique decomposition $x=m+n$, $m\in M$, $n\in N$. Then $Px=m$, $Qx=n$, and hence
\begin{equation*}
	R(\lambda)x=(\lambda I-A)^{-1}m-B(I-\lambda B)^{-1}n\in M+B(N)=M+T^{-1}(N).
\end{equation*}
Conversely, let $u\in M$ and $v\in T^{-1}(N)$. Since $(\lambda I-A)^{-1}\colon M\to M$, $(I-\lambda B)^{-1}\colon N\to N$ are surjective, we have
\begin{equation*}
	\operatorname{ran}((\lambda I-A)^{-1})=M\qquad\text{and}\qquad \ran(B(I-\lambda B)^{-1})=B(N)=T^{-1}(N). 
\end{equation*}
Thus, there exist $m\in M$ and $n\in N$ such that $(\lambda I-A)^{-1}m=u$ and $-B(I-\lambda B)^{-1}n=v$. Let $x=m+n$. Then $Px=m$, $Qx=n$ and hence
\begin{equation*}
	R(\lambda)x=(\lambda I-A)^{-1}m-B(I-\lambda B)^{-1}n=u+v.
\end{equation*}
Thus, \eqref{eq:commonrange} holds.

Define 
\begin{equation}\label{eq:defE}
E=M+T^{-1}(N)
\end{equation}
 and
\begin{equation*}
	\Pi(\lambda)=R(\lambda)(\lambda I-T),\qquad \lambda\in\mathbb T_\epsilon.
\end{equation*}
Since $\lambda I-T$ is surjective for $\lambda\in\mathbb T_\epsilon$, the last relation, \eqref{eq:commonrange} and~\eqref{eq:defE}, yield
\begin{equation}\label{eq:samerange}
	\operatorname{ran}\Pi(\lambda)=\operatorname{ran}R(\lambda)=E,\qquad\lambda\in\mathbb T_\epsilon.
\end{equation}
Using~\eqref{eq:rightinveq}, we have for $\lambda\in\mathbb T_\epsilon$,
\begin{equation*}
\Pi^2(\lambda)=	\Pi(\lambda)\Pi(\lambda)=R(\lambda)[(\lambda I-T)R(\lambda)](\lambda I-T)=R(\lambda)(\lambda I-T)=\Pi(\lambda).
\end{equation*}
Thus, $\Pi(\lambda)$ is a projection for $\lambda\in\mathbb T_\epsilon$. Since $\Pi(\lambda)$ is a  projection, its range $E=\operatorname{ran}\Pi(\lambda)$ (see \eqref{eq:samerange}) is a closed subspace of~$X$ and $\Pi(\lambda)|_E=I_E$ for $\lambda\in\mathbb T_\epsilon$. This, together with $\operatorname{ran}R(\mu)=E$ for $\mu\in\mathbb T_\epsilon$ (see ~\eqref{eq:commonrange} and \eqref{eq:defE}), implies 
\begin{equation*}
	\Pi(\lambda)R(\mu)=R(\mu)\qquad\text{whenever $\lambda$, $\mu\in\mathbb T_\epsilon$}.
\end{equation*}
Hence
\begin{equation}\label{eq:lamuid}
	R(\lambda)(\lambda I-T)R(\mu)=\Pi(\lambda)R(\mu)=R(\mu)\qquad\text{whenever $\lambda$, $\mu\in\mathbb T_\epsilon$}.
\end{equation}
By~\eqref{eq:rightinveq}, we have for $\lambda$, $\mu\in\mathbb T_\epsilon$,
\begin{equation*}
	R(\lambda)(\mu I-T)R(\mu)=R(\lambda)I=R(\lambda).
\end{equation*}
Subtracting~\eqref{eq:lamuid} from the last equation, we obtain for $\lambda$, $\mu\in\mathbb T_\epsilon$,
\begin{align*}
R(\lambda)-R(\mu)&=R(\lambda)(\mu I-T)R(\mu)-R(\lambda)(\lambda I-T)R(\mu)\\
&=R(\lambda)[\,(\mu I-T)-	(\lambda I-T)\,]R(\mu)\\
&=R(\lambda)(\mu-\lambda)R(\mu).
\end{align*}
Hence
\begin{equation}
	R(\lambda)-R(\mu)=(\mu-\lambda)R(\lambda)R(\mu)\qquad\text{for all $\lambda$, $\mu\in\mathbb T_\epsilon$ };
\end{equation}
that is, the resolvent equation~\eqref{eq:reseq} on $\Omega=\mathbb T_\epsilon$ also holds.

\textit{Sufficiency.} Suppose that there exists $\epsilon\in(0,1)$ such that 
$\mathbb T_\epsilon\subseteq\rho_r(T)$ and $T$ admits a right resolvent function~$R$ on~$\mathbb T_\epsilon$. We will show that~$T$ is generalized hyperbolic. The resolvent equation~\eqref{eq:reseq}, together with the continuity of~$R$, implies for $\lambda\in\mathbb T_\epsilon$,
\begin{equation}
	R'(\lambda)=\lim_{h\to0}\frac{R(\lambda+h)-R(\lambda)}{h}=\lim_{h\to0}\bigl(-R(\lambda+h)R(\lambda)\bigr)=-R^2(\lambda).
\end{equation}
Thus, $R$ is holomorphic on the open annulus~$\mathbb T_\epsilon$. Therefore, $R$ can be expanded into a Laurent series
\begin{equation}\label{eq:laurexp}
R(\lambda)=\sum_{n=-\infty}^\infty C_n\lambda^n,
\quad 1-\epsilon<|\lambda|<1+\epsilon,
\end{equation}
where the unique coefficients $C_n\in L(X)$, $n\in\mathbb Z$, can be expressed in the form of a contour integral
\begin{equation}\label{eq:intrep}
C_n=\frac{1}{2\pi i}\int_{|\lambda|=\rho}\lambda^{-n-1}R(\lambda)\,d\lambda,\qquad n\in\mathbb Z,
\end{equation}
where $\rho\in(1-\epsilon,1+\epsilon)$ is arbitrary and the circle $|\lambda|=\rho$ is positively oriented~\cite{Hil}. Hence
\begin{equation}\label{eq:coeffest}
	\|C_n\|\leq K(\rho)\rho^{-n}\qquad\text{whenever $\rho\in(1-\epsilon,1+\epsilon)$ and $n\in\mathbb Z$}, 
\end{equation}
where
\begin{equation*}
	K(\rho)=\max_{|\lambda|=\rho}\|R(\lambda)\|,\qquad \rho\in(1-\epsilon,1+\epsilon).
	\end{equation*}
	 Substituting~\eqref{eq:laurexp} into~\eqref{eq:hfri}, we obtain
\begin{equation*}
\sum_{n=-\infty}^\infty(C_{n-1}-TC_n)\lambda^n=I\qquad\text{whenever $1-\epsilon<|\lambda|<1+\epsilon$}.
\end{equation*}
From this, in view of the uniqueness of the coefficients of the Laurent series, we have 
\begin{equation*}
C_{-1}-TC_0=I
\end{equation*}
and 
\begin{equation*}
C_{n-1}-TC_n=0\qquad\text{for $n\neq0$}.
\end{equation*}
Hence
\begin{gather}
	C_0=-T^{-1}(I-C_{-1}),\label{eq:crel1}\\
		C_n=T^{-n}C_0=-T^{-n-1}(I-C_{-1}),\qquad n=1,2,\dots\label{eq:crel2}
\end{gather}
and
\begin{equation}\label{eq:crel3}
C_{-n}=T^{n-1} C_{-1},\qquad n=1,2,\dots
\end{equation}
Let $\mu\in\mathbb T_\epsilon$ be fixed. Choose $\rho\in(1-\epsilon,|\mu|)$. Define $P=C_{-1}$. From \eqref{eq:reseq} and~\eqref{eq:intrep} with $n=-1$, we obtain
\begin{align*}
PR(\mu)&=\frac{1}{2\pi i}\int_{|\lambda|=\rho} R(\lambda)R(\mu)\,d\lambda=\frac{1}{2\pi i}\int_{|\lambda|=\rho}\frac{ R(\lambda)-R(\mu)}{\mu-\lambda}\,d\lambda\\
&=\frac{1}{2\pi i}\int_{|\lambda|=\rho}\frac{ R(\lambda)}{\mu-\lambda}\,d\lambda-\frac{R(\mu)}{2\pi i}\int_{|\lambda|=\rho}\frac{ 1}{\mu-\lambda}\,d\lambda.
\end{align*}
By Cauchy's theorem, the second integral is zero because $\rho<|\mu|$ implies that the only singularity of the integrand $\lambda=\mu$ lies outside the disc $|\lambda|\leq\rho$. For the calculation of the first integral, note that, on the circle $|\lambda|=\rho$, we have
$$
\biggl|\frac{\lambda}{\mu}\biggr|=\frac{\rho}{|\mu|}<1.
$$
Hence
$$
\frac{1}{\mu-\lambda}=\frac{1}{\mu\bigl(1-\lambda/\mu\bigr)}
=\frac{1}{\mu}\sum_{k=0}^\infty\biggl(\frac{\lambda}{\mu}\biggr)^k=\sum_{k=0}^\infty\frac{\lambda^k}{\mu^{k+1}}
$$
whenever $|\lambda|=\rho$ and the convergence of the last series is uniform on the contour $|\lambda|=\rho$. Consequently,
\begin{align*}
	PR(\mu)=\frac{1}{2\pi i}\int_{|\lambda|=\rho}R(\lambda)\sum_{k=0}^\infty\frac{\lambda^k}{\mu^{k+1}}\,d\lambda
	=\frac{1}{2\pi i}\sum_{k=0}^\infty\mu^{-k-1}\int_{|\lambda|=\rho}\lambda^k R(\lambda)\,d\lambda.
\end{align*}
From this, using~\eqref{eq:intrep}, we have
\begin{equation*}
PR(\mu)=\sum_{k=0}^\infty\mu^{-k-1}C_{-k-1}=	\sum_{n=1}^\infty\mu^{-n}C_{-n},\qquad \mu\in\mathbb T_\epsilon.
\end{equation*}
On the other hand,~\eqref{eq:laurexp} yields
\begin{equation*}
PR(\mu)=\sum_{n=-\infty}^\infty PC_n\mu^n,
\quad \mu\in\mathbb T_\epsilon.
\end{equation*}
In view of the uniqueness of the Laurent coefficients, we have
\begin{align}
PC_{-n}&=C_{-n},\qquad n\geq1,\label{eq:first comp}\\
PC_n&=0,\qquad n\geq0.\label{eq:seconcomp}
\end{align}
Since $C_{-1}=P$, relations~\eqref{eq:crel1}--\eqref{eq:seconcomp} imply 
\begin{align}
P^2&=P,\label{eq:prop1}\\
TP&=PTP,\label{eq:prop2}\\
T^{-1}Q&=QT^{-1}Q.\label{eq:prop3}
\end{align}
where $Q=I-P$. Indeed, we have
\begin{gather*}
	P^2=PC_{-1}\overset{\text{\eqref{eq:first comp}}}{=}C_{-1}=P,\\
	TP=TC_{-1}\overset{\text{\eqref{eq:crel3}}}{=}C_{-2}\overset{\text{\eqref{eq:first comp}}}{=}PC_{-2}
	\overset{\text{\eqref{eq:crel3}}}{=}PTC_{-1}=PTP,\\
	PT^{-1}(I-P)=PT^{-1}(I-C_{-1})\overset{\text{\eqref{eq:crel1}}}{=}-PC_0\overset{\text{\eqref{eq:seconcomp}}}{=}0,
\end{gather*}
the last equality yielding
\begin{equation*}
QT^{-1}Q=(I-P)T^{-1}(I-P)=T^{-1}(I-P)-PT^{-1}(I-P)=T^{-1}(I-P)=T^{-1}Q.
\end{equation*}
Thus, \eqref{eq:prop1}--\eqref{eq:prop3} hold. By~\eqref{eq:prop1}, $P$ and hence $Q=I-P$ is a projection. Their ranges, 
\begin{equation*}
M=\operatorname{ran}P\qquad\text{and}\qquad N=\operatorname{ran}Q,	
\end{equation*}
are closed subspaces of~$X$ such that $X=M\oplus N$. It is easily seen that~\eqref{eq:prop2} and~\eqref{eq:prop3} are equivalent to
\begin{equation*}
	T(M)\subseteq M\qquad\text{and}\qquad T^{-1}(N)\subseteq N.
\end{equation*}
It remains to show that the spectral radii of the restrictions
\begin{equation*}
	A=T|_M\in L(M)
\qquad\text{and}\qquad
B=T^{-1}|_N\in L(N)
\end{equation*}
are less than one. By~\eqref{eq:coeffest} and~\eqref{eq:crel3}, we have for $x\in M$ and $n\geq0$,
\begin{align*}
\|A^n x\|&=\|T^n Px\|=\|T^n C_{-1}x\|=\|C_{-(n+1)}x\|\\
&\leq\|C_{-(n+1)}\|\|x\|\leq K(\rho)\rho^{n+1}\|x\|,
\end{align*}
which implies
\begin{equation*}
\|A^n\|=\sup_{0\neq x\in M}\frac{\|A^n x\|}{\|x\|}\leq\rho K(\rho)\rho^{n},
\end{equation*}
where $\rho\in(1-\epsilon,1+\epsilon)$ is arbitrary. From this, by the spectral radius formula,
\begin{equation*}
r(T|_M)=r(A)=\lim_{n\to\infty}\root{n}\of{\|A^n\|}\leq \rho.
\end{equation*}
Letting $\rho\to1-\epsilon$ in the last inequality, we conclude that
\begin{equation*}
r(T|_M)\leq1-\epsilon<1.
\end{equation*}
From~\eqref{eq:coeffest} and~\eqref{eq:crel2}, we obtain in a similar manner for $x\in N$ and $n\geq2$,
\begin{align*}
\|B^n x\|&=\|T^{-n} Qx\|=\|T^{-n}(I-P)x\|=\|T^{-n}(I-C_{-1})x\|\\
&=\|-C_{n-1}x\|\leq\|C_{n-1}\|\|x\|\leq K(\rho)\rho^{-(n-1)}\|x\|,
\end{align*}
which yields
\begin{equation*}
\|B^n\|=\sup_{0\neq x\in N}\frac{\|B^n x\|}{\|x\|}\leq\rho K(\rho)\rho^{-n},
\end{equation*}
where $\rho\in(1-\epsilon,1+\epsilon)$ is arbitrary. Hence
\begin{equation*}
r(T^{-1}|_N)=r(B)=\lim_{n\to\infty}\root{n}\of{\|B^n\|}\leq\frac{1}{\rho}.
\end{equation*}
Letting $\rho\to1+\epsilon$, we obtain
\begin{equation*}
r(T^{-1}|_N)\leq\frac{1}{1+\epsilon}<1.
\end{equation*}
Thus, $T$ is generalized hyperbolic.
\end{proof}

\section{Proofs of Theorems~\ref{thm:A} and~\ref{thm:B}}

For the reader's convenience, we summarize some known results which will be needed in the proof of Theorem~\ref{thm:A}.

 Let $\mathbb N$ be the set of positive integers. Given a complex Banach space~$X$, denote by $\ell^1(\mathbb N,X)$ and 
 $\ell^\infty(\mathbb N, X)$ the space of those sequences $(x_n)_{n\in\mathbb N}$ in~$X$ which are absolutely summable and bounded, respectively. With their standard norms,
\begin{align*}
 \|x\|_1&:=\sum_{n=1}^\infty\|x_n\|,\qquad x=(x_n)_{n\in\mathbb N}\in \ell^1(\mathbb N,X),\\
\|x\|_\infty&:=\sup_{n\in\mathbb N}\|x_n\|,\qquad x=(x_n)_{n\in\mathbb N}\in \ell^\infty(\mathbb N,X),
\end{align*}
both $\ell^1(\mathbb N,X)$ and 
 $\ell^\infty(\mathbb N, X)$ are Banach spaces. It is known that
$$
c_0(\mathbb N,X):=\bigl\{\,x=(x_n)_{n\in\mathbb N}\in \ell^\infty(\mathbb N,X):\lim_{n\to\infty}\|x_n\|=0\,\bigr\}
$$
is a closed subspace of $\ell^\infty(\mathbb N,X)$. Therefore, with the induced norm, $c_0(\mathbb N,X)$ is also a Banach space. In the special case $X=\mathbb C$, we shall write $\ell^\infty=\ell^\infty(\mathbb N,\mathbb C)$ and $c_0=c_0(\mathbb N,\mathbb C)$ for brevity.

The following classical result, due to Phillips~\cite{Ph}, will play a fundamental role in the proof of Theorem~\ref{thm:A}.

\begin{theorem}[{\cite[Theorem~3.2.20, p.~301]{M}}]\label{thm:phillips}
The closed subspace~$c_0$ is not complemented in~$\ell^\infty$.
\end{theorem}
 
Now suppose that $M$ is a closed subspace of the Banach space~$X$. 
 The \emph{quotient space $X/M$} is defined by
\begin{equation*}
X/M=\{\,[x]:x\in X\,\},\qquad\text{where $[x]:=x+M$}.
\end{equation*}
The set $X/M$, equipped with the natural vector-space operations
\begin{equation*}
[x]+[y]=[x+y], \qquad \lambda[x]=[\lambda x],\quad\qquad x, y\in X,\,\,\lambda\in\mathbb C,
\end{equation*}
and the \emph{quotient norm}
\begin{equation*}
\|[x]\|:=\inf_{m\in M}\|x-m\|,
\end{equation*}
is a Banach space (see \cite[Theorem~1.7.7, p.~53]{M}). It is known (see \cite[Proposition~1.7.12, p.~54]{M}) that the \emph{quotient map} $q_M\colon X\to X/M$, defined by
\begin{equation*}
q_M(x)=[x],\qquad x\in X,
\end{equation*}
is a bounded linear surjection satisfying
\begin{equation*}
\ker q_M=M \qquad\text{and}\qquad \|q_M\|\leq 1.
\end{equation*}

Let $E$ and $F$ be complex Banach spaces. The \emph{external direct sum} $E\oplus_1 F$ is the product vector space $$E\times F=\{\,(x,y):x\in E,\,y\in F\,\},$$ equipped with the coordinatewise vector operations and the $1$-norm 
$$
\|(x,y)\|_{E\oplus_1 F}=\|x\|_E+\|y\|_F.
$$ 
If $E$ and~$F$ are Banach spaces, then $E\oplus_1 F$ is also a Banach space (see \cite[Sec.~1.8]{M}).

Now let $U\in L(E,F)$, where $L(E,F)$ denotes the Banach space of bounded linear operators with the operator norm. We say that \emph{$U$ admits a bounded linear right inverse} if there exists $V\in L(F,E)$ such that $UV=I_F$.

From \cite[Theorem~2.12, p.~39]{B}, applied to the quotient map~$q_M$, we obtain the following useful criterion for a closed subspace~$M$ of a Banach space~$X$ to be complemented.

\begin{proposition}\label{prop:complemented}
Let $M$ be a closed subspace of a complex Banach space~$X$. Then $M$ is complemented in~$X$ if and only if the quotient map $q_M\colon X\to X/M$ admits a bounded linear right inverse.	
\end{proposition}

Proposition~\ref{prop:complemented} implies that Phillips' theorem (Theorem~\ref{thm:phillips}) can be formulated equivalently as follows.

\begin{proposition}\label{prop:phillipsref}
The quotient map $q_{c_0}\colon\ell^\infty\rightarrow \ell^\infty/c_0$ has no bounded linear right inverse.	
\end{proposition}

We are in a position to give a proof of Theorem~A. In view of Theorem~\ref{thm:specban} and
Corollary~\ref{cor:ghrh}, it is sufficient to find an invertible operator~$T$ on a Banach space such that
\begin{equation*}
\sigma_s(T)\cap\mathbb T=\emptyset
\qquad\text{but}\qquad
\sigma_r(T)\cap\mathbb T\neq\emptyset.
\end{equation*}
The construction is based on the distinction between surjectivity and
right invertibility.
We start from the quotient map
$q_{c_0}:\ell^\infty\longrightarrow \ell^\infty/c_0$ which, by Proposition~\ref{prop:phillipsref}, has no bounded linear right inverse. Using~$q_{c_0}$, we construct an operator~$S$ on a suitable Banach space~$Y$
which retains this failure of right invertibility, while its small
scalar perturbations are surjective. More precisely, $S$ has no bounded linear right inverse, whereas
$S+\alpha I$ is surjective for every $\alpha\in\mathbb{C}$ with $|\alpha|<1$.
We then encode~$S$ into an invertible operator~$T$ on $X=Y\oplus_1 Y$.
The operator~$T$ is defined so that, for every
$\lambda\in\mathbb{T}$, the surjectivity of $\lambda I-T$ reduces
to the surjectivity of an operator of the form $S+\alpha_\lambda I$ with $|\alpha_\lambda|<1$. Therefore, the preceding property of~$S$ 
guarantees that $\lambda I-T$ is surjective for every $\lambda\in\mathbb{T}$, and hence
$\sigma_s(T)\cap\mathbb{T}=\emptyset$. At the same time, the structure preserves the obstruction to
right invertibility: a bounded linear right inverse of $I-T$ would yield a bounded linear right inverse of~$S$. Since no such right inverse of~$S$ exists, we obtain that $1\in\sigma_r(T)$.
Thus, the construction transfers the gap between surjectivity and right invertibility from the quotient map~$q_{c_0}$, through the auxiliary operator~$S$, to an invertible operator~$T$ whose
surjectivity spectrum avoids the unit circle while its right spectrum
meets it. Now we present the detailed construction.

\begin{proof}[Proof of Theorem~\ref{thm:A}]
Write $Q:=\ell^\infty/c_0$ for brevity and let $q_{c_0}\colon\ell^\infty\to Q$ be the quotient map from Proposition~\ref{prop:phillipsref}.
Let
\begin{equation*}	
E:=\ell^1\!\bigl(\N,\ell^\infty\bigr)
\qquad\text{and}\qquad
Y:=Q\oplus_1E.
\end{equation*}
Let $B\colon E\to E$ be the backward shift,
\[
Be=B(e_1,e_2,\ldots):=(e_2,e_3,\ldots),\qquad e=(e_1,e_2,\ldots)\in E.
\]
Define $\pi_1\colon E\to\ell^\infty$ by
\begin{equation*}
\pi_1 e=\pi_1(e_1,e_2,\ldots):=e_1,\qquad e=(e_1,e_2,\ldots)\in E.
\end{equation*}
Finally, define $S\colon Y\to Y$ by
\begin{equation*}	
S(x,e):=\bigl(q_{c_0}(\pi_1 e),Be\bigr),\qquad x\in Q,\,e\in E.
\end{equation*}
It is easily seen that $q_{c_0}$, $\pi_1$, and $B$ are bounded linear contractions\footnote{Their operator norms are not greater than one.}. This implies that~$S$ is linear and for all $x\in Q$ and $e\in E$,
\begin{align*}
	\|S(x,e)\|_Y
&=\|q_{c_0}(\pi_1 e                )\|_Q+\|Be\|_E
\leq\|q_{c_0}\|\|\pi_1\|\|e\|_E+\|B\|\|e\|_E\\
&\le 2\|e\|_E\leq 2\|x\|_Q+2\|e\|_E
=2\|(x,e)\|_Y.
\end{align*}
Thus, $S\in L(Y)$.

\begin{claim}
For every $\alpha\in\C$ with $|\alpha|<1$, the operator $S+\alpha I\in L(Y)$ is surjective.
\end{claim}
\begin{proof}We distinguish two cases.
\vskip3pt

\emph{Case 1.\,} $\alpha=0$.

Let $(u,f)\in Y$ so that $u\in Q$ and $f=(f_1,f_2,\dots)\in E$. Since $q_{c_0}$ is surjective, there exists $a\in\ell^\infty$ such that $q_{c_0}(a)=u$. Define
\begin{equation*}
e:=(a,f_1,f_2,\ldots)\in E.
\end{equation*}
Then $\pi_1e=a$ and hence $Be=f$. Hence
\begin{equation*}
S(0,e)=\bigl(q_{c_0}(\pi_1 e),Be\bigr)=\bigl(q_{c_0}(a),f\bigr)=(u,f).
\end{equation*}
Since $(u,f)\in Y$ was arbitrary, $S$ is surjective.
\vskip3pt

\emph{Case 2.\,} $0<|\alpha|<1$.

Let $F\colon E\to E$ be the forward shift,
\begin{equation*}
Fe=F(e_1,e_2,\ldots):=(0,e_1,e_2,\ldots),\qquad e=(e_1,e_2,\ldots)\in E.
\end{equation*}
Then $BF=I_E$ and $\|F\|=1$. Since $|\alpha|\|F\|<1$, $I+\alpha F$ is invertible by the Neumann series. Given $(u,f)\in Y$, define
\begin{equation*}	
g:=(I+\alpha F)^{-1}f,
\qquad
e:=Fg,
\qquad
x:=\frac{u-q_{c_0}(\pi_1 e)}{\alpha}.
\end{equation*}
From the last equation, we have
\begin{equation*}
q_{c_0}(\pi_1 e)+\alpha x=u,	
\end{equation*}
while the first two equations, combined with $BF=I$, yield
\begin{equation*}
Be+\alpha e=BFg+\alpha Fg=g+\alpha Fg=(I+\alpha F)g=f.	
\end{equation*}
Hence
\begin{equation*}
(S+\alpha I)(x,e)
=\bigl(q_{c_0}(\pi_1e)+\alpha x,\;Be+\alpha e\bigr)=(u,f).	
\end{equation*}
Since $(u,f)\in Y$ was arbitrary, $S+\alpha I$ is surjective.
\end{proof}

\begin{claim}
The operator $S$ has no bounded linear right inverse.
\end{claim}

\begin{proof}
Suppose, by way of contradiction, that there exists an operator $R\in L(Y)$ such that 
\begin{equation}\label{eq:Srightinv2}
SR=I_Y.	
\end{equation}
Let $P_1\colon Y\to Q$ and $P_2\colon Y\to E$ be the canonical projections,
\begin{equation*}
P_1(x,e):=x\qquad\text{and}\qquad P_2(x,e):=	e,\qquad x\in Q,\,e\in E.
\end{equation*}
Define $\iota\colon Q\to Y$ by
\begin{equation*}
	\iota x=(x,0),\qquad x\in Q,
\end{equation*}
where $0$ denotes the zero element of~$E$. Evidently,
\begin{equation}\label{eq:iotarel}
P_1\iota=I_Q.	
\end{equation}
By the definition of $S$, we have
\begin{equation}\label{eq:Sprojrel}
P_1S=q_{c_0}\pi_1 P_2.
\end{equation}
Then $\sigma:=\pi_1P_2R\iota:Q\to\ell^\infty$, as a product of bounded linear operators, is bounded linear. Furthermore,
\begin{equation*}	
q_{c_0}\sigma=q_{c_0}\pi_1P_2 R\iota
\overset{\text{\eqref{eq:Sprojrel}}}{=}P_1SR\iota
\overset{\text{\eqref{eq:Srightinv2}}}{=}P_1\iota
\overset{\text{\eqref{eq:iotarel}}}{=}I_Q.
\end{equation*}
Thus, $\sigma$ is a bounded linear right inverse of~$q_{c_0}$ contradicting Proposition~\ref{prop:phillipsref}.
\end{proof}
Define
\begin{equation*}
	X:=Y\oplus_1Y
\end{equation*}
and $T\colon X\to X$ by
\begin{equation}\label{eq:Tdef}
T(x,y):=(y,x-3Sy), \qquad x\in Y,\,\,y\in Y.
\end{equation}
Since~$S$ is bounded and linear, it is easy to show that~$T$ has the same properties.
\begin{claim}
$T\in L(X)$ is invertible.
\end{claim}

\begin{proof}

Define $W\colon X\to X$ by
\begin{equation*}
W(x,y):=(y+3Sx,x),\qquad x\in Y,\,\,y\in Y.
\end{equation*}
Clearly, $W$ is linear and bounded. For every $x\in Y$ and $y\in Y$,
\begin{equation*}
T W(x,y)=T(y+3Sx,x)=(x,y),
\end{equation*}
and
\begin{equation*}
W T(x,y)=W(y,x-3Sy)=(x,y).
\end{equation*}
Since $TW=WT=I_X$, $T$ is invertible and $T^{-1}=W$.

\end{proof}

\begin{claim}
For every $\lambda\in\T$, the operator $\lambda I-T$ is surjective.
\end{claim}

\begin{proof}
Fix $\lambda\in\mathbb T$ and let $(u,v)\in X$ be arbitrary. Define
\begin{equation}\label{eq:alphalambda}
\alpha_\lambda:=\frac{\lambda^2-1}{3\lambda}.
\end{equation}
Since $|\lambda|=1$, we have
\begin{equation*}
|\alpha_\lambda|
=\frac{|\lambda^2-1|}{3}
\leq\frac{|\lambda|^2+1}{3}
=\frac{2}{3}<1.
\end{equation*}
By Claim~1, this implies that $S+\alpha_\lambda I$ is surjective. Therefore, there exists $x\in Y$ such that
\begin{equation}\label{eq:perturbedS}
\left(S+\alpha_\lambda I\right)x
=\frac{v+\lambda u+3Su}{3\lambda}.
\end{equation}
Define
\begin{equation}\label{eq:ydef}
y:=\lambda x-u.
\end{equation}
Then
$$
\begin{aligned}
(\lambda I-T)(x,y)
&\overset{\text{\eqref{eq:Tdef}}}{=}(\lambda x-y,\lambda y-x+3Sy)\\
&\overset{\text{\eqref{eq:ydef}}}{=}\bigl(u,\lambda(\lambda x-u)-x+3S(\lambda x-u)\bigr)\\
&=\bigl(u,(\lambda^2-1)x+3\lambda Sx-\lambda u-3Su\bigr)\\
&=\biggl(u, 3\lambda\biggl(S+\frac{\lambda^2-1}{3\lambda}I\biggr)x-\lambda u-3Su\biggr)\\
&\overset{\text{\eqref{eq:alphalambda}}}{=}\bigl(u,3\lambda\bigl(S+\alpha_\lambda I\bigr)x-\lambda u-3Su\bigr)\\
&\overset{\text{\eqref{eq:perturbedS}}}{=}(u,v).
\end{aligned}
$$
Since $\lambda\in\mathbb T$ and $(u,v)\in X$ were arbitrary,
$\lambda I-T$ is surjective for every $\lambda\in\mathbb T$.
\end{proof}

\begin{claim}
$1\in\sigma_r(T)$.
\end{claim}

\begin{proof}
Suppose, by way of contradiction, that $1\notin\sigma_r(T)$; that is, there exists $R\in L(X)$ such that
\begin{equation}\label{eq:Rid}
	(I-T)R=I_X.
\end{equation}
Define $J\colon Y\to X$ by
\begin{equation*}
	Jz=(0,3z),\qquad z\in Y.
\end{equation*}
Let $P_1,P_2\colon X\to Y$ be the canonical projections,
\begin{equation*}
P_1(x,y):=x\qquad\text{and}\qquad P_2(x,y):=	y,\qquad x\in Y,\,y\in Y.
\end{equation*}
By~\eqref{eq:Tdef}, we have
\begin{equation*}
(I-T)(x,y)=(x-y,-x+y+3Sy),\qquad x\in Y,\,y\in Y.
\end{equation*}
Let $z\in Y$ be arbitrary. Applying the last identity to
\begin{equation*}
	(x,y)=RJz=(P_1RJz,P_2RJz),
\end{equation*}
we obtain
\begin{equation}\label{eq:firstid}
(I-T)RJz=(P_1RJz-P_2 RJz,-P_1RJz+P_2 RJz+3SP_2RJz).
\end{equation}
On the other hand, by~\eqref{eq:Rid}, we have
\begin{equation}\label{eq:secondid}
(I-T)RJz=Jz=(0,3z).	
\end{equation}
Since $z\in Y$ was arbitrary,  comparing the corresponding coordinates on the right hand side of Eqs.~\eqref{eq:firstid} and~\eqref{eq:secondid}, we obtain the operator equations
\begin{gather*}
P_1RJ-P_2RJ=0,\\
-P_1RJ+P_2RJ+3SP_2RJ=3I_Y.	
\end{gather*}
Adding the last two equations, we find that
\begin{equation}\label{eq:Srightinv}
	SP_2RJ=I_Y.
\end{equation}
Since $J$, $R$ and $P_2$ are bounded linear operators, $L:=P_2RJ$ is also bounded and linear. Moreover, in view of~\eqref{eq:Srightinv}, $L$ is a bounded linear right inverse of~$S$ contradicting Claim~2.
\end{proof}

We can now easily complete the proof of Theorem~\ref{thm:A}. By Claim~3, $T\in GL(X)$. By Claim~4, $\sigma_s(T)\cap\mathbb T=\emptyset$. By Theorem~\ref{thm:specban}, this guarantees that $T$ has the shadowing property. By Claim~5, $\sigma_r(T)\cap\mathbb T\neq\emptyset$. By Corollary~\ref{cor:ghrh}, this shows that $T$ is not generalized hyperbolic.

\end{proof}

Now we give a proof of Theorem~\ref{thm:B}.

\begin{proof}[Proof of Theorem~\ref{thm:B}]
According to Theorem~\ref{thm:specuniexp}, generalized hyperbolicity implies the shadowing property even for Banach space operators. It remains to prove the converse. Suppose that $X$ is a separable complex Hilbert space and $T\in GL(X)$ has the shadowing property. We will show that $T$ is generalized hyperbolic. By
Corollary~\ref{cor:spechilb}, 
$\sigma_r(T)\cap\mathbb T=\emptyset$.
Since $\rho_r(T)$ is an open subset of~$\mathbb C$ and contains the compact set $\mathbb T$,
there exists $\epsilon\in(0,1)$ such that
$
\overline{\mathbb T_\epsilon}\subset\rho_r(T)
$.
We will show that the function $\operatorname{nul}(\lambda I-T)=\dim\ker(\lambda I-T)$ is locally constant in $\lambda\in\rho_r(T)$.\footnote{Although the local constancy of the nullity can be deduced from \cite[Lemmas~9.4 and~9.5]{A}, we give a detailed argument for completeness.}
 (The symbol $\operatorname{dim}$ denotes the Hilbert dimension, the cardinality of an orthonormal basis.) To prove the claim, fix an arbitrary $\lambda_0\in\rho_r(T)$. Since  $\lambda_0 I-T$ is right invertible, there exists $R_0\in L(X)$ such
that
\begin{equation}\label{eq:zerorightinv}
(\lambda_0 I-T)R_0=I.
\end{equation}
It is known that if $|\lambda-\lambda_0|\|R_0\|<1$, then 
\begin{equation}\label{eq:defU}
	U(\lambda):=I+(\lambda-\lambda_0)R_0
\end{equation}
is invertible and its inverse $[U(\lambda)]^{-1}$ is given by the Neumann series. Therefore, we can define
\begin{equation}\label{eq:defR}
R(\lambda)=R_0 [U(\lambda)]^{-1}\qquad\text{for $\lambda\in D(\lambda_0,\delta_0)$},	
\end{equation}
where
\begin{equation*}
D(\lambda_0,\delta_0):=\{\,\lambda\in\mathbb C:|\lambda-\lambda_0|<\delta_0\,\}\qquad\text{with $\delta_0:=\frac{1}{\|R_0\|}$}.	
\end{equation*}
For $\lambda\in D(\lambda_0,\delta_0)$, 
\begin{align*}
(\lambda I-T)R(\lambda)&\overset{\text{\eqref{eq:defR}}}{=}[(\lambda_0 I-T)+(\lambda-\lambda_0)I]R_0[U(\lambda)]^{-1}\\
&=[(\lambda_0 I-T)R_0+(\lambda-\lambda_0)R_0][U(\lambda)]^{-1}\\
&\overset{\text{\eqref{eq:zerorightinv}}}{=}[I+(\lambda-\lambda_0)R_0][U(\lambda)]^{-1}
\overset{\text{\eqref{eq:defU}}}{=}U(\lambda)[U(\lambda)]^{-1}=I.
\end{align*}
Thus,
\begin{equation}\label{eq:rightinvR}
(\lambda I-T)R(\lambda)=I,\qquad \lambda\in D(\lambda_0,\delta_0).	
\end{equation}
We claim that
\begin{equation}\label{eq:defP}
P(\lambda):=R(\lambda)(\lambda I-T)\quad\text{is a projection for $\lambda\in D(\lambda_0,\delta_0)$},
\end{equation}
and
\begin{equation}\label{eq:Pkerid}
\ker P(\lambda)=\ker(\lambda I-T),\qquad \lambda\in D(\lambda_0,\delta_0).
\end{equation}
Indeed,
\begin{equation*}
	P^2(\lambda)=P(\lambda)P(\lambda)=R(\lambda)[(\lambda I-T)R(\lambda)](\lambda I-T)
	\overset{\text{\eqref{eq:rightinvR}}}{=}
	R(\lambda)(\lambda I-T)=P(\lambda)
\end{equation*}
for $\lambda\in D(\lambda_0,\delta_0)$. Thus, \eqref{eq:defP} holds. By~\eqref{eq:defP}, $\ker(\lambda I-T)\subseteq\ker P(\lambda)$, while
\begin{equation*}
\lambda I-T=I(\lambda I-T)\overset{\text{\eqref{eq:rightinvR}}}{=}(\lambda I-T)R(\lambda)(\lambda I-T)\overset{\text{\eqref{eq:defP}}}{=}(\lambda I-T)P(\lambda)
\end{equation*}
implies the converse inclusion $\ker P(\lambda)\subseteq\ker(\lambda I-T)$ for $\lambda\in D(\lambda_0,\delta_0)$. Thus, \eqref{eq:Pkerid} also holds. From~\eqref{eq:defP} and~\eqref{eq:Pkerid}, it follows that
\begin{equation}\label{eq:dirP}
	X=\ker P(\lambda)\oplus\operatorname{ran}P(\lambda)=\ker(\lambda I-T)\oplus\operatorname{ran}P(\lambda),\qquad\lambda\in D(\lambda_0,\delta_0).
\end{equation}
Since both $[U(\lambda)]^{-1}$ and $\lambda I-T$ are surjective for $\lambda\in D(\lambda_0,\delta_0)$, \eqref{eq:defP} and~\eqref{eq:defR} imply
\begin{equation*}
	\operatorname{ran}P(\lambda)=\operatorname{ran}R(\lambda)=\operatorname{ran}R_0=:E,\qquad\lambda\in D(\lambda_0,\delta_0).
\end{equation*}
This, together with~\eqref{eq:dirP}, yields
\begin{equation}\label{eq:commoncompl}
X=\ker(\lambda I-T)\oplus E,\qquad\lambda\in D(\lambda_0,\delta_0).
\end{equation}
As a range of a projection, $E=\operatorname{ran}P(\lambda_0)$ is a closed subspace of~$X$.
 By \cite[Corollary~3.2.16, p.~299]{M}, \eqref{eq:commoncompl} implies that all  subspaces $\ker(\lambda I-T)$, $\lambda\in D(\lambda_0,\delta_0)$, are isomorphic to the same quotient space $X/E$, and hence they are isomorphic to each other.
It is known (see \cite[Corollary~5.5.2.6, p.~194]{C}) that isomorphic Hilbert spaces have the same Hilbert dimension. Hence
\begin{equation*}
	\dim\ker(\lambda I-T)=\dim\ker(\lambda_0 I-T),\qquad \lambda\in D(\lambda_0,\delta_0).
\end{equation*}
Since $\lambda_0\in\rho_r(T)$ was arbitrary, this proves that $\operatorname{nul}(\lambda I-T)$ is locally constant in $\lambda\in\rho_r(T)$. The annulus~$\mathbb T_\epsilon$ is connected, therefore, the locally constant function $\operatorname{nul}(\lambda I-T)$ is in fact constant for $\lambda\in\mathbb T_\epsilon$. By Proposition~\ref{prop:sapostol}, applied with $\Omega=\mathbb T_\epsilon$, we conclude that~$T$ admits a right resolvent function on~$\mathbb T_\epsilon$. Thus, Theorem~\ref{thm:C} applies and guarantees that $T$ is generalized hyperbolic.
\end{proof}

\section*{Acknowledgements}
The author wishes to thank Professor Davor Dragi\v{c}evi\'c for many helpful discussions.


\begin{thebibliography}{99}

\bibitem{All}
G. R. Allan,
\textit{Holomorphic vector-valued functions on a domain of holomorphy}, J.~Lond. Math. Soc. \textbf{4} (1967), 509--513.

\bibitem{AMV}
M. B. Antunes, G. E. Mantovani, and R. Var{\~a}o,
\textit{Chain recurrence and positive shadowing in linear dynamics},
J. Math. Anal. Appl. \textbf{506} (2022), 125622.

\bibitem{A}
C.~Apostol, L.~A.~Fialkow, D.~A.~Herrero, and D.~Voiculescu,
Approximation of Hilbert Space Operators, Vol.~II,
Research Notes in Mathematics, vol.~102,
Pitman, Boston, 1984.

\bibitem{BCDFP}
N. C. Bernardes Jr., B. M. Caraballo, U. B. Darji, V. V. F\'avaro, and A.~Peris,
\textit{Generalized hyperbolicity, stability and expansivity for operators on locally convex spaces}, J. Func. Anal. \textbf{288} (2025), 110696.

 \bibitem{Ber}
 N. Bernardes Jr., P. R. Cirilo, U. B. Darji, A. Messaoudi, and E. R. Pujals,
\textit{Expansivity and shadowing in linear dynamics}, J. Math. Anal. Appl. \textbf{461} (2018), 796--816.
 
 \bibitem{BeMe}
 N. Bernardes Jr. and A. Messaoudi, 
\textit{Shadowing and structural stability for operators}, Ergodic Theory Dynam. Systems \textbf{41} (2021), 961--980.
 
 \bibitem{BP}
 N. C. Bernardes Jr. and A. Peris, 
\textit{On shadowing and chain recurrence in linear dynamics}, Adv. Math. \textbf{441} (2024), 109539.
 
 \bibitem{B}
H.~Brezis,
Functional Analysis, Sobolev Spaces and Partial Differential Equations,
Universitext, Springer, New York, 2011.
 
 \bibitem{Cir}
 P.~R. Cirilo, B. Gollobit, and E.~Pujals, 
\textit{Dynamics of generalized hyperbolic linear operators}, Adv. Math. \textbf{387} (2021), 107830.
 
 \bibitem{C}
C. Constantinescu,
$C^*$-Algebras, Vol.~4: Hilbert Spaces,
North-Holland Mathematical Library,
Elsevier, Amsterdam, 2001. 
 
\bibitem{Dani}
E.~D'Aniello, U.~B.~Darji, and M.~Maiuriello,
\textit{Generalized hyperbolicity and shadowing in $L^p$ spaces}, J.~Differential Equations \textbf{298} (2021), 68--94.

\bibitem{DAnMa}
E.~D'Aniello and M.~Maiuriello,
\textit{On the spectrum of weighted shifts}, Rev. R. Acad. Cienc. Exactas F\'{\i}s. Nat. Ser. A Math. RACSAM \textbf{117} (2023), 1--19.

\bibitem{DP}
D.~Dragi\v{c}evi\'c and M.~Pituk,
\textit{Duality between shadowing and uniform expansivity in linear dynamics},
Trans. Amer. Math. Soc., 2026,
\href{https://doi.org/10.1090/tran/9879}{DOI 10.1090/tran/9879}.

 \bibitem{H}
D.~A.~Herrero,
Approximation of Hilbert Space Operators, Vol.~I,
Research Notes in Mathematics, Vol.~72,
Pitman, Boston, 1982.
 
 \bibitem{Hil}
E.~Hille and R. S.~Phillips,
 Functional Analysis and Semi-groups, American Mathematical Society, Providence, RI, 1957.

\bibitem{Iva}
 S.~Ivanov, \emph{On holomorphic relative inverses of operator-valued functions}, Pacific J.~Math. \textbf{78} (1978), 345--358.
 
\bibitem{LM}
K.~Lee and C.~A.~Morales,
\textit{Hyperbolicity, shadowing, and bounded orbits},
Qual. Theory Dyn. Syst. \textbf{21} (2022), 61,
\href{https://doi.org/10.1007/s12346-022-00588-9}{DOI 10.1007/s12346-022-00588-9}. 

 \bibitem{LM1}
J. Lee and C. A. Morales,
\textit{Hyperbolicity, shadowing, and convergent operators},
Monatsh. Math. \textbf{202} (2023), 541--554.

\bibitem{LM2}
J. Lee and C. A. Morales,
\textit{Shadowing lemmas for noninvertible operators and semigroups},
Rev. Mat. Complut. \textbf{38} (2025), 531--548.
 
 \bibitem{M}
R. E. Megginson,
An Introduction to Banach Space Theory,
Graduate Texts in Mathematics, Vol.~183,
Springer--Verlag, New York, 1998.

\bibitem{MNST}
A.~Messaoudi, J.~T.~Neto, M.~Saavedra, I.~Tsokanos,
\textit{On generalized hyperbolicity, stability and shadowing for linear operators},
\href{https://arxiv.org/abs/2608.17021}{arXiv:2608.17021}, 2026.

\bibitem{Mo}
C.~A. Morales,
Recent Topics on Linear Dynamics,
\href{https://arxiv.org/abs/2502.20666}{arXiv:2502.20666}, 2025.

\bibitem{Pal}
K. Palmer, Shadowing in Dynamical Systems, Theory and Applications, Kluwer, Dordrecht, 2000.
 
 \bibitem{Ph}
R.~S. Phillips,
\textit{On linear transformations},
Trans. Amer. Math. Soc. \textbf{48} (1940), 516--541.

\bibitem{Pil}
S. Yu. Pilyugin, 
Shadowing in Dynamical Systems, Lecture Notes in Mathematics, Vol.~1706, Springer--Verlag, Berlin, 1999.

\bibitem{P}
M.~Pituk,
\textit{Spectral characterization of shadowing for linear operators on Hilbert spaces},
\href{https://arxiv.org/abs/2511.12272}{arXiv:2511.12272}, 2025.

\bibitem{T}
A. E. Taylor,
Introduction to Functional Analysis,
John Wiley \& Sons, New York, 1958.

\end{thebibliography}
\end{document}